\documentclass[11pt]{amsart}
\usepackage{amsthm}
\usepackage{mathrsfs}
\usepackage{enumerate}
\usepackage{mathrsfs}
\usepackage{amssymb,amsmath,amsthm,color}
\usepackage{blindtext}
\usepackage{esint}
\usepackage{hyperref}
\usepackage[pdftex]{graphicx}

\theoremstyle{plain}

\newtheorem{theorem}{Theorem}[section]

\newtheorem{lemma}[theorem]{Lemma}

\newtheorem{remark}[theorem]{Remark}

\newcommand{\R}{\mathbb R}
\newcommand{\Z}{\mathbb Z}
\newcommand{\Na}{\mathbb N}

\newcommand{\C}{{\mathbb C}}

\renewcommand{\H}{\mathbb H}

\renewcommand{\Im}{\operatorname{Im}}
\newcommand{\tr}{\operatorname{tr}}

\title{Kato-Ponce estimates for Fractional Sublaplacians}
\date{\today}

\author[L. Fanelli, L. Roncal]{Luca Fanelli and  Luz Roncal}

\address{Luca Fanelli: Ikerbasque $\&$ Departmento de Matem\'aticas, Universidad del Pa\'is Vasco/Euskal Herriko Unibertsitatea (UPV/EHU), Aptdo. 644, 48080, Bilbao, Spain}
\email{luca.fanelli@ehu.es}
\address{Luz Roncal: Basque Center for Applied Mathematics (BCAM), 48009, Bilbao, Spain and Ikerbasque, Basque Foundation for Science, 48011 Bilbao, Spain}
\email{lroncal@bcamath.org}

\begin{document}

\begin{abstract}
We give a proof of commutator estimates for fractional powers of the sublaplacian on the Heisenberg group. Our approach is based on pointwise and $L^p$ estimates involving square fractional integrals and  Littlewood-Paley square functions. 
\end{abstract}

	\maketitle

\section{Introduction}
\label{sec:intro}

In \cite{KP}, Kato and Ponce proved the well known commutator estimate
$$
\|J^s(fg)-fJ^sg\|_{L^p(\R^n)}\lesssim\|J^sf\|_{L^p(\R^n)}\|g\|_{L^{\infty}(\R^n)}+\|\partial f\|_{L^\infty(\R^n)}\|J^{s-1}g\|_{L^{\infty}(\R^n)},
$$
for $1<p<\infty$, and $s>0$,  where $J^s:=(1-\Delta)^{s/2}$, $\partial=(\partial_1,\cdots, \partial_n)$ and $f,g\in \mathcal{S}(\R^n)$. Closely related to this, we have the following estimate by Kenig, Ponce, and Vega in \cite{KPV}
$$
\|(-\Delta)^{s/2}(fg)-f(-\Delta)^{s/2}g-g(-\Delta)^{s/2}f\|_{L^p(\R^n)}\lesssim \|(-\Delta)^{s_1/2}f\|_{L^{p_1}(\R^n)}\|(-\Delta)^{s_2/2}g\|_{L^{p_2}(\R^n)},
$$
where $s=s_1+s_2$, $0<s,s_1,s_2<1$, $\frac1p=\frac{1}{p_1}+\frac{1}{p_2}$ and $1<p,p_1,p_2<\infty$.
The above estimates naturally arise in several different contexts. In particular, they turn out to be fundamental to close fixed point arguments in Sobolev spaces for some nonlinear dispersive PDE's. This motivates the investigation about the validity of commutator estimates in different geometries than the Euclidean setting. 
 
In the recent paper \cite{Ma}, Maalaoui provided a pointwise estimate for commutators involving fractional powers of the sublaplacian on Carnot groups of homogeneous dimension $Q$. The result in \cite{Ma} is rather general, and includes the case of the fractional powers of the sublaplacian on the Heisenberg group. In the present manuscript, we give an alternative proof in this case, which is more strictly related to the geometric structure of the Heisenberg group. 

Our approach, which is inspired in the proof by D'Ancona in \cite{D} for the Euclidean case, is based on the study of nontangential square functions as crucial tools for the proof. In addition, our strategy makes use of both
 non-conformal and conformal harmonic extensions associated to the sublaplacian. We found the study of these ingredients of independent interest. They motivated us to take a chance to revisit the result by \cite{Ma} in this particular case and provide also weighted versions of the result.

Before stating our main results, we need to introduce the geometric and functional setting.
A remarkable way to characterize nonlocal operators such as $(-\Delta)^{s/2}$ is via a weighted \textit{Dirichlet-to-Neumann} map of a extension problem. This approach is present in the literature since the 1950's, with the paper by Huber \cite{Hu}. Closely related and containing the same circle of ideas, we find the work by Muckenhoupt and Stein \cite{MuSt}. We also mention the extension procedure introduced by Molchanov and Ostrovskii in \cite{MO} within a context of probability, see also the related work by Spitzer \cite{Sp} and the more recent by Kolsrud \cite{Ko}.

In particular, the landmark work by Caffarelli and Silvestre \cite{Caffarelli-Silvestre}, in which they studied the extension problem associated to the Laplacian on $\R^n$, and realized the fractional power $(-\Delta)^{s/2}$ as the map taking Dirichlet data to Neumann data, has been a rich source of development in the study of nonlocal operators in the last few years, specially from the point of view of partial differential equations. 

Fractional powers of Laplacians also occur naturally in conformal geometry and scattering theory.  Chang-Gonz\'alez \cite{CG} showed that the fractional order Paneitz operators $P_{\gamma}$ arising in the work of Graham and Zworski \cite{GZ} in conformal geometry coincide with $(-\Delta)^{s/2}$ when the conformally compact Einstein manifold is taken to be the hyperbolic space. Later, Frank et al. \cite{FGMT} studied the extension problem associated to the sublaplacian $\mathcal{L}$ on the Heisenberg group $\H^n$. Unlike the case of $\R^n$, where $(-\Delta)^{s/2}$ are conformally invariant, in the context of Heisenberg groups $\mathcal{L}^s$, defined as the map taking Dirichlet to Neumann data in \eqref{eq:ST} below, are not. Hence, conformally invariant fractional powers of the sublaplacian, denoted by $\mathcal{L}_s$, are more relevant from a geometrical point of view than the pure fractional powers $\mathcal{L}^s$, see \cite{B,BFM, FL}. 

Let $\H^n:=\C^n\times \R$ denote the $(2n+1)$ dimensional Heisenberg group (see Section \ref{sec:tool} for a brief review of the group structure).
For $s>0$, given a function $f\in C_0^{\infty}(\H^n\times \R^+)$, the extension problem for $\mathcal{L}^s$ consists of finding $U\in C_0^{\infty}(\H^n\times \R^+)$ such that
\begin{equation}
\label{eq:epHnoc}
\begin{cases}
\big(\partial_{\rho\rho}+\frac{1-2s}{\rho}\partial_{\rho}-\mathcal{L}\big)U((z,t),\rho)=0 \qquad ((z,t),\rho) \in \H^n\times \R^+,\\
U((z,t),0)=g(z,t), \qquad (z,t)\in \H^n.
\end{cases}
\end{equation}
The extension problem for general second order partial differential operators has been studied by Stinga-Torrea \cite{ST}. The sublaplacian on $\H^n$ lies within this general theory and then it is shown that
\begin{equation}
\label{eq:ST}
\mathcal{L}^sg=c_s\lim_{\rho\to0} \rho^{1-2s}\partial_{\rho}U.
\end{equation}
We mention that when we consider $-\Delta$ and $\R^n$ instead of $\mathcal{L}$ and $\H^n$, then we are dealing with the extension problem for $(-\Delta)^s$ in \cite{Caffarelli-Silvestre}.

For $s>0$, the extension problem for the sublaplacian $\mathcal{L}_s$ on $\H^n$  consists of finding $U\in C_0^{\infty}(\H^n\times \R^+)$ such that
\begin{equation}
\label{eq:epH}
\begin{cases}
\big(\partial_{\rho\rho}+\frac{1-2s}{\rho}\partial_{\rho}+\frac14\rho^2\partial_{tt}-\mathcal{L}\big)U((z,t),\rho)=0 \qquad ((z,t),\rho) \in \H^n\times \R^+,\\
U((z,t),0)=f(z,t), \qquad (z,t)\in \H^n.
\end{cases}
\end{equation}
Note that the latter extension problem is different from the problem \eqref{eq:epHnoc} due to the appearance of the extra term $\frac14\rho^2\partial_t^2$. Indeed, if we consider $\H^n$ as the boundary of the Siegel's upper half space $\Omega_{n+1}$, then the extension problem \eqref{eq:epH} occurs naturally. Using this connection, Frank et al.  \cite{FGMT} showed that for $f\in C_0^{\infty}(\H^n)$ there is a unique solution of the above equation which satisfies 
$$
\mathcal{L}_sf=c_s\lim_{\rho\to0} \rho^{1-2s}\partial_{\rho}U.
$$

We have defined the conformally and non conformally invariant fractional powers $\mathcal{L}_s$ and $\mathcal{L}^s$, respectively, via the corresponding extension problems. Other equivalent definitions are available and moreover it can be checked, see Subsection \ref{sub:fractionalSub}, that the operators $\mathcal{L}_s $ and $ \mathcal{L}^s $ are equivalent in $ L^p(\H^n) $, i.e., there exist $c, C>0$ such that
$$
c\|\mathcal{L}^sf\|_{L^p}\le \|\mathcal{L}_sf\|_{L^p}\le C \|\mathcal{L}^sf\|_{L^p},
\qquad
1<p<\infty.
$$

Our main result is the following.
  \begin{theorem}
  \label{thm:main}
  Let $n\ge1$. Assume that $s,s_1,s_2$ and $p,p_1,p_2$ satisfy
  $$
  s=s_1+s_2, \qquad s_j\in (0,1/4), \qquad \frac{1}{p}=\frac{1}{p_1}+\frac{1}{p_2}, \qquad \frac{2Q}{Q+2s_j}<p_j<\infty.
  $$
  Then for all $u,v\in \mathcal{S}(\H^n)$ we have 
  \begin{equation}
  \label{eq:1}
  \|\mathcal{L}_s(u v)-u\mathcal{L}_sv-v\mathcal{L}_su\|_{L^p}\lesssim \|\mathcal{L}_{s_1}u\|_{L^{p_1}}\|\mathcal{L}_{s_2}v\|_{L^{p_2}}
  \end{equation}
  and 
  \begin{equation}
  \label{eq:2}
  \|\mathcal{L}^s(u v)-u\mathcal{L}^sv-v\mathcal{L}^su\|_{L^p}\lesssim \|\mathcal{L}^{s_1}u\|_{L^{p_1}}\|\mathcal{L}^{s_2}v\|_{L^{p_2}}.
\end{equation}
  Moreover, for $w_j\in A_{q_j}$, where $1<q_j=p_j\big(\frac12+\frac{s_j}{Q}\big)$,
    \begin{equation}
  \label{eq:3}
  \|\mathcal{L}_s(u v)-u\mathcal{L}_sv-v\mathcal{L}_su\|_{L^p(w_1^{p/p_1}w_2^{p/p_2})}\lesssim \|\mathcal{L}_{s_1}u\|_{L^{p_1}(w_1)}\|\mathcal{L}_{s_2}v\|_{L^{p_2}(w_2)}
\end{equation}
  and 
  \begin{equation}
  \label{eq:4}
  \|\mathcal{L}^s(u v)-u\mathcal{L}^sv-v\mathcal{L}^su\|_{L^p(w_1^{p/p_1}w_2^{p/p_2})}\lesssim \|\mathcal{L}^{s_1}u\|_{L^{p_1}(w_1)}\|\mathcal{L}^{s_2}v\|_{L^{p_2}(w_2)}.
\end{equation}
  \end{theorem}
  
  \begin{remark}
  We notice that Theorem \ref{thm:main} is providing weighted versions of the Kato-Ponce inequalities for fractional sublaplacians in the Heisenberg group, which were missing in \cite{Ma}.
  \end{remark}
  
We will follow the ideas in \cite{D}, which in turn are inspired by \cite[Chapter V, \S 6.12]{St}. The proof of our theorem will use analogue tools as the ones utilized in the Euclidean case. Nevertheless, in the Heisenberg group, such tools will be sometimes not explicitly available and we will have to work them out. 
We define the \textit{square fractional integral} as
  $$
  \mathcal{D}_su(x):=\Big(\int_{\H^n}\frac{|u(xy^{-1})-u(x)|^2}{|y|^{Q+4s}}\,dy\Big)^{1/2}, \quad 0<s<1/2,\quad x\in \H^n,
  $$
 where $xy^{-1}$ is the right translation by $y^{-1}$ on the Heisenberg group, see Subsection \ref{sec:Heisenberg}, and $ Q = 2n+2$ is the \textit{homogeneous dimension} of $\H^n$.
One of the crucial steps in the proof is a pointwise estimate for the square fractional integrals $ \mathcal{D}_{s}$ by the so-called $g_{\lambda}^*$-function, defined in terms of the Poisson semigroup associated to the non-conformally invariant harmonic extension, i.e., to the problem \eqref{eq:epHnoc} for $s=1/2$.  Let $X_1,\ldots, X_n, Y_1,\ldots, Y_n$, $T$, be basis for the Lie algebra of left-invariant vector fields on $\H^n$ (see Subsection~\ref{sub:Sub}). Let 
$$
\nabla =(X_1, \ldots, X_{n}, Y_1,\ldots, Y_{n}, \partial_{\rho}),
$$ 
 we define the \textit{Littlewood nontangential square function} $g_{\lambda}^*$ as
\begin{equation*}
g_{\lambda}^*(u)(x):=\Big(\int_0^{\infty}\int_{\H^n}\Big(\frac{\rho}{\rho+|y|}\Big)^{\lambda Q}\rho^{1-Q}|\nabla U(xy^{-1},\rho)|^2\,dy\,d\rho\Big)^{1/2},\quad x\in \H^n,
\end{equation*}
where $U(x,\rho)$ is the non-conformal harmonic extension of $u(x)$ in the upper half space. We will prove the following.
  \begin{theorem}
  \label{thm:pointwise}
  Let $n\ge 1$, $0<s<1/2$ and $\lambda<1+\frac{2s}{Q}$. Then
  $$
   \mathcal{D}_{s}u(x)\le \Lambda(n,s)g_{\lambda}^*(\mathcal{L}^su)(x)
  $$
  uniformly on $u\in \mathcal{S}(\H^n)$ and $x\in \H^n$, where $\Lambda(n,s)>0$ is a constant depending only on $n$ and $s$.
  \end{theorem}

\subsection*{Structure of the paper} We start gathering some well known facts about the Heisenberg group and fractional powers of the sublaplacian in Section \ref{sec:tool}. In Section \ref{sec:technical} we provide some technical results that will be needed to prove the main results. In particular, mapping properties for the square function and the nontangential square function are shown, and a mean value theorem for subharmonic functions on $\H^n\times \R^+$ is stated. Finally, the proofs of Theorems  \ref{thm:pointwise} and \ref{thm:main} are presented, respectively, in Sections \ref{sec:proofTh2} and \ref{sec:proofTh1}.

\subsection*{Acknowledgements} L. R. was partially supported by the Basque Government through the BERC 2018-2021 program, by the Spanish Ministry of Economy and Competitiveness through BCAM Severo Ochoa excellence accreditation SEV-2017-2018 and through project PID2020-113156GB-I00, by the project RYC2018-025477-I, and by Ikerbasque.

The authors whish to thank Adriano Pisante for addressing us some useful references about the mean value property related to hypoelliptic operators on Carnot groups, and Sundaram Thangavelu for helpful clarifications.
 
 \section{The Heisenberg group and fractional powers of the sublaplacian}
 \label{sec:tool}
 
 Let us first introduce some definitions and set up notations concerning the Heisenberg group. We refer the reader to the book of G. B. Folland \cite{Fo}, although we closely follow the notations used in \cite{STH}.  We also warn the reader that our notation and certain definitions may be slightly different from those used by others.
 
\subsection{Fourier transform on the Heisenberg group}
\label{sec:Heisenberg}

Let $\H^n=\C^n\times \R$ be the $(2n+1)$ dimensional Heisenberg group, which is the nilpotent Lie group of step two whose underlying manifold is $\R^{2n+1}$ equipped with the group law
$$
(z,t)(z',t')=\Big(z+z',t+t'+\frac12\Im z\cdot \overline{z'}\Big),
$$
where $z,z'\in \C^n$ and $t,t' \in \R$. 
Identifying $ \H^n $ with $ \R^{2n+1} $ and considering coordinates $ (x,u,t) $ we can write the group law as
\begin{equation}
\label{eq:symplectic}
(x,y,t)(x',y',t') = \Big(x+x',y+y',t+t'+\frac{1}{2}(x\cdot y'-x'\cdot y)\Big), 
\end{equation}
where $x,x',y,y'\in \R^n$ and $t,t' \in \R$. 
Note that  $ \Im \big((x+iy)\cdot(x'-iy')\big)= y\cdot x'-y'\cdot x = [(x,y)(x',y')] $ is the symplectic form on $ \R^{2n}$.

For each $\lambda\in \R^* = \R\setminus\{0\}$, we have an irreducible  unitary representation $\pi_{\lambda}$ of $\H^n$ realized on $ L^2(\R^n).$ The action of $ \pi_\lambda(z,t) $ on $ L^2(\R^n) $ is explicitly given by
$$
\pi_\lambda(z,t)\varphi(\xi) = e^{i\lambda t} e^{i(x\cdot \xi+\frac12 x\cdot y)}\varphi(\xi+y) 
$$
where $ \varphi \in L^2(\R^n) $ and $ z = x+iy$. By a theorem of Stone and Von Neumann, any irreducible unitary representation of $ \H^n $ which acts as
$ e^{i\lambda t} \operatorname{Id} $ at the center of the Heisenberg group is unitarily equivalent to $ \pi_\lambda $.  In view of this, there are representations of $ \H^n $ which are realized on the Fock spaces and equivalent to $ \pi_\lambda$. We will not use these representations and refer the reader to \cite{Fo} for details. 

The group Fourier transform of a function  $f\in L^1(\H^n)$ is the operator-valued function defined, for each $\lambda\in \R^*$, by
\begin{equation*}
\widehat{f}(\lambda) := \pi_{\lambda}(f)= \int_{\H^n}f(z,w)\pi_{\lambda}(z,w)\,dz\,dw.
\end{equation*}
With an abuse of language, we will call the group Fourier transform just the Fourier transform.
Observe that for each $\lambda$, $\widehat{f}(\lambda)$ is an operator acting on $L^2(\R^n)$. When $f\in L^1\cap L^2(\H^n)$, it can be shown that $\widehat{f}(\lambda)$ is a Hilbert-Schmidt operator and the Plancherel theorem holds
\begin{equation}
\label{eq:Plancherel}
\int_{\H^n}|f(z,t)|^2\,dz\,dt=\frac{2^{n-1}}{\pi^{n+1}}
\int_{-\infty}^{\infty}\|\widehat{f}(\lambda)\|_{\operatorname{HS}}^2|\lambda|^n\,d\lambda,
\end{equation}
where $\|\cdot\|_{\operatorname{HS}}$ is the Hilbert-Schmidt norm given by $\|T\|^2_{\operatorname{HS}}=\tr(T^* T)$, for $T$ a bounded operator, being $T^*$ the adjoint operator of $T$.
By polarizing the Plancherel identity we get the Parseval formula
$$
\int_{\H^n} f(z,t) \overline{g(z,t)} dz dt= \frac{2^{n-1}}{\pi^{n+1}}
\int_{-\infty}^{\infty}  \tr(\widehat{f}(\lambda)\widehat{g}(\lambda)^*) |\lambda|^n\,d\lambda.
$$

Let $ f^\lambda $ stand for the inverse Fourier transform of $ f $ in the \textit{central variable} $t$
\begin{equation*}
f^\lambda(z) = \int_{-\infty}^\infty f(z,t) e^{i\lambda t} \,dt.
\end{equation*}
By taking the Euclidean Fourier transform of $ f^\lambda(z)$  in the variable $\lambda$, we obtain
\begin{equation}
\label{eq:lambdaFT}
f(z,t)=\frac{1}{2\pi}\int_{-\infty}^{\infty}e^{-i\lambda t}f^{\lambda}(z)\,d\lambda.
\end{equation}
By the definition of $ \pi_\lambda(z,t) $ and $ \widehat{f}(\lambda) $ it is easy to see that
\begin{equation*}
\widehat{f}(\lambda) = \int_{\C^n} f^\lambda(z) \pi_\lambda(z,0) dz.
\end{equation*}
The operator which takes  a function $ g $ on $ \C^n $ into the operator
$$ \int_{\C^n} g(z) \pi_\lambda(z,0) dz $$ is called the Weyl transform of $ g $ and is denoted by $ W_\lambda(g)$. Thus $ \widehat{f}(\lambda) = W_\lambda(f^\lambda)$.

Let us recall that the convolution of $ f $ with $ g $ on $\H^n$ is defined by
$$
f*g(x) = \int_{\H^n} f(xy^{-1})g(y) \,dy, \quad x,y\in \H^n.
$$
 With  $x=(z,t)$ and $y=(z',t')$, in view of \eqref{eq:symplectic}, we have that $(z',t')=(-z',-t')$ and the above takes the form
$$
f*g(z,t) = \int_{\H^n} f\big((z,t)(-z',-t')\big)g(z',t') \,dz' \,dt'.
$$
A simple computation shows that
\begin{equation*}
(f*g)^\lambda(z) = \int_{\C^n} f^\lambda(z-z')g^\lambda(z') e^{\frac{i}{2}\Im(z\cdot \bar{z'})} \,dz'.
\end{equation*}
The convolution appearing on the right hand side is called the $ \lambda$-twisted convolution and is denoted by $ f^\lambda*_\lambda g^\lambda(z)$.

\subsection{Spectral theory in the Heisenberg group}
\label{sec:Hermite}

For $\lambda\in \R^*$ and $ \alpha \in \Na^n $, we introduce the family of Hermite functions
$$
\Phi_\alpha^\lambda(x) = |\lambda|^{\frac{n}{4}}\Phi_\alpha(\sqrt{|\lambda|}x), \quad x\in \R^n.
$$
Here, $\Phi_\alpha $ is the normalized Hermite function on $\R^n$ which is an eigenfunction of the Hermite operator $H = -\Delta+|x|^2 $ with eigenvalue $ (2|\alpha|+n)$, see \cite[Chapter 1.4]{STH}. The system is an orthonormal basis for $L^2(\R^n) $. In terms of $\Phi_\alpha^\lambda$ we have the identity
$$
\|\widehat{f}(\lambda)\|_{\operatorname{HS}}^2=\sum_{\alpha\in \mathbb{N}^n}\|\widehat{f}(\lambda)\Phi_\alpha^\lambda\|_{L^2}^2
$$ 
and hence Plancherel \eqref{eq:Plancherel} takes the form
$$
\int_{\H^n}|f(z,t)|^2\,dz\,dt=\frac{2^{n-1}}{\pi^{n+1}}
\int_{-\infty}^{\infty}\sum_{\alpha\in \mathbb{N}^n}\|\widehat{f}(\lambda)\Phi_\alpha^\lambda\|_{L^2}^2|\lambda|^n\,d\lambda.
$$

We can write the spectral decomposition of the scaled Hermite operator $ H(\lambda) =  -\Delta+|\lambda|^2 |x|^2 $ as
\begin{equation}
\label{eq:spectralHermite}
H(\lambda) = \sum_{k=0}^\infty (2k+n)|\lambda| P_k(\lambda), \quad \lambda \in \R^*,
\end{equation}
where $ P_k(\lambda) $ are the (finite-dimensional) orthogonal projections defined on $  L^2(\R^n)$ by
$$
P_k(\lambda)\varphi = \sum_{|\alpha| =k} (\varphi,\Phi_\alpha^\lambda) \Phi_\alpha^\lambda,
$$
where $ \varphi \in L^2(\R^n) $ and $(\cdot, \cdot)$ is the inner product in $L^2(\R^n)$.

On the other hand, we define the scaled Laguerre functions of type $ (n-1)$
\begin{equation*}
\varphi_k^\lambda(z) = L_k^{n-1}\Big(\frac12 |\lambda||z|^2\Big)e^{-\frac14 |\lambda||z|^2}.
\end{equation*}
Here $ L_k^{n-1} $ are the Laguerre polynomials of type $ (n-1)$, see \cite[Chapter 1.4] {STH} for the definition and properties. It happens that $\{\varphi_k^\lambda\}_{k=0}^{\infty}$ forms an orthogonal basis for the subspace consisting of radial functions in $ L^2(\C^n).$  

The so-called special Hermite expansion of a function
$ g $ defined on $ \C^n $ written in its compact form reads as
$$
g(z)=(2\pi)^{-n} |\lambda|^n \sum_{k=0}^\infty  g*_\lambda \varphi_k^\lambda(z).
$$
 The connection betweeen the Hermite projections $P_k(\lambda)$ and the Laguerre functions $\varphi_k^\lambda$, via the Weyl transform, is given by the following important formula
\begin{equation}
\label{eq:WeylLaguerre}
W_\lambda(\varphi_k^\lambda) = (2\pi)^n |\lambda|^{-n} P_k(\lambda).
\end{equation}
Observe that, in particular, for any function $f$ on $\H^n$, we have the expansion
\begin{equation*}
f^{\lambda}(z)=(2\pi)^{-n} |\lambda|^n \sum_{k=0}^\infty  f^{\lambda}*_\lambda \varphi_k^\lambda(z).
\end{equation*}

\subsection{The sublaplacian}
\label{sub:Sub}

Let us now define the sublaplacian on the Heisenberg group.
A basis for the Lie algebra $ \mathfrak{h}_n $ of left-invariant vector fields on $\H^n$ is given by
\begin{equation}
\label{vfields}
X_j = \frac{\partial}{\partial{x_j}}+\frac{1}{2}y_j \frac{\partial}{\partial t},\qquad Y_j = \frac{\partial}{\partial{y_j}}-\frac{1}{2}x_j \frac{\partial}{\partial t}, \quad j=1,2,\ldots, n, \qquad  T = \frac{\partial}{\partial t}.
\end{equation}
It is easily checked that the only non-trivial Lie brackets in $ \mathfrak{h}_n $ are given by $ [ X_j, Y_j] =T$ as all other brackets vanish.  The Kohn-Laplacian on $\H^n$ is the second order operator 
$$ 
\mathcal{L} = -\sum_{j=1}^n ( X_j^2 +Y_j^2), 
$$
known as the sublaplacian. It falls in the class of operators of the type \textit{sums of squares of vector fields}. Though not elliptic, this operator shares several properties with its counterpart $ \Delta $ on $ \R^n$. 

The group $ \H^n $  admits a family of automorphisms indexed by $ \R_+ $ and given by the non-isotropic \textit{Heisenberg dilations} 
$$ 
\delta_{\lambda} (z,t) = ( \lambda z, \lambda^2 t), \quad \lambda>0, \quad (z,t)\in \H^n. 
$$
A function $u: \H^n\to \R$ is said \textit{homogenous of degree $k\in \Z$} if for every $\lambda>0$
$$
u\circ \delta_{\lambda}=\lambda^k u.
$$
With respect to these dilations, the vector fields $ X_j, Y_j, T $ are homogeneous of degree one and $ \mathcal{L}$ is homogeneous of degree two. A fundamental solution $\Gamma(z,t)$ for $ \mathcal{L} $ is given by 
\begin{equation*}
\Gamma(z,t)=c_Q|(z,t)|^{-Q+2} 
\end{equation*}
where $ Q = 2n+2$ is the homogeneous dimension of $\H^n$, $ |(z,t)|^4 = |z|^4+16 t^2$, and $c_Q>0$ is a number depending only on $Q$. This was found by Folland \cite{F}, see also \cite{S}.  The distance function 
\begin{equation*}
d: (z,t) \mapsto |(z,t)| 
\end{equation*}
is the Koranyi norm, which is homogeneous of degree one with respect to the dilations $ \delta_{\lambda}$.

The spectral decomposition of the sublaplacian is achieved via the special Hermite expansion introduced in the previous subsection.  The action of the Fourier transform on functions of the form $\mathcal{L}f$ and $Tf$ are given by
$$
(\mathcal{L}f)^{\widehat{}}(\lambda)
= \widehat{f}(\lambda)H(\lambda), \qquad (Tf)^{\widehat{}}(\lambda)=-i\lambda\widehat{f}(\lambda).
$$
If $ L_\lambda $ is the operator defined by the relation $ (\mathcal{L}f)^\lambda = L_\lambda f^\lambda $ then it follows that
$$  W_\lambda(L_\lambda f^\lambda) =  W_\lambda(f^\lambda) H(\lambda). $$
Recalling the spectral decomposition of $ H(\lambda) $ given in \eqref{eq:spectralHermite} and the identity \eqref{eq:WeylLaguerre} we obtain
$$   L_\lambda f^\lambda(z) = (2\pi)^{-n} \sum_{k=0}^\infty (2k+n)|\lambda|  f^\lambda*_\lambda \varphi_k^\lambda(z).$$
Thus, by taking the Fourier transform in the variable $\lambda$ in  \eqref{eq:lambdaFT}, the spectral decomposition of the sublaplacian is given by
\begin{equation*}
\mathcal{L}f(z,t) = (2\pi)^{-n-1} \int_{-\infty}^\infty \Big( \sum_{k=0}^\infty (2k+n)|\lambda|  f^\lambda*_\lambda \varphi_k^\lambda(z)\Big) e^{-i\lambda t}  |\lambda|^n d\lambda.
\end{equation*}

\subsection{Fractional powers of the sublaplacian}
\label{sub:fractionalSub}

The fractional powers of the sublaplacian $\mathcal{L}^s$ defined in the introduction via the extension problem \eqref{eq:epHnoc} can be equivalently defined via the spectral decomposition
$$ \mathcal{L}^sf(z,t) = (2\pi)^{-n-1} \int_{-\infty}^\infty \Big( \sum_{k=0}^\infty \big((2k+n)|\lambda|\big)^s  f^\lambda*_\lambda \varphi_k^\lambda(z)\Big) e^{-i\lambda t}  |\lambda|^n d\lambda.$$
Note that $ (\widehat{\mathcal{L}^sf})(\lambda) = \widehat{f}(\lambda)H(\lambda)^s.$

On the other hand, the operators $ \mathcal{L}_s $ are also defined for $0\le s < (n+1) $ by
\begin{equation}
\label{eq:modifFract1}
\mathcal{L}_sf(z,t) = (2\pi)^{-n-1} \int_{-\infty}^\infty \Big( \sum_{k=0}^\infty (2|\lambda|)^s \frac{\Gamma(\frac{2k+n}{2}+\frac{1+s}{2})}{ \Gamma(\frac{2k+n}{2}+\frac{1-s}{2})} f^\lambda*_\lambda \varphi_k^\lambda(z)\Big) e^{-i\lambda t}|\lambda|^n d\lambda.
\end{equation} 
As mentioned in the introduction, the operators $\mathcal{L}_s$ occur naturally in the context of CR geometry and scattering theory on the Heisenberg group: when we identify $\H^n$ as the boundary of the Siegel's upper half space in $ \C^{n+1} $, they have the important property of being conformally invariant. 
In short, \eqref{eq:modifFract1} means that $ \mathcal{L}_s $ is the operator (see \cite[(1.33)]{BFM})
\begin{equation*}
\mathcal{L}_s:=(2|T|)^s\frac{\Gamma\big(\frac{\mathcal{L}}{2|T|}+\frac{1+s}{2}\big)}
{\Gamma\big(\frac{\mathcal{L}}{2|T|}+\frac{1-s}{2}\big)}.
\end{equation*}
Thus $\mathcal{L}_s$ corresponds to  the spectral multiplier
\begin{equation*}
(2|\lambda|)^s\frac{\Gamma\big(\frac{2k+n}{2}+\frac{1+s}{2}\big)}
{\Gamma\big(\frac{2k+n}{2}+\frac{1-s}{2}\big)}, \quad k\in \Na.
\end{equation*}
Note that $ \mathcal{L}_1 = \mathcal{L} $. It is known that $ \mathcal{L}_s$ also has an explicit fundamental solution, see e.g. \cite[page 530]{CH}, in contrast with $ \mathcal{L}^s $, whose fundamental solution cannot be written down explicitly. 

It can be checked that the operators $\mathcal{L}_s $ and $ \mathcal{L}^s $ are equivalent in $ L^p(\H^n) $, $1<p<\infty$, i.e.,
\begin{equation}
\label{eq:equiv}
c\|\mathcal{L}^sf\|_{L^p}\le \|\mathcal{L}_sf\|_{L^p}\le C \|\mathcal{L}^sf\|_{L^p},
\end{equation}
for some $c,C>0$.
Indeed, it suffices to prove that the operator $\mathcal{L}^{-s} \mathcal{L}_s$ is bounded on $L^p(\H^n)$. In order to conclude the latter, all we need to do is to check that the multiplier
$$   
M= \sum_{k=0}^\infty  (2k+n)^{-s} \frac{\Gamma((2k+n+1+s)/2)}{\Gamma((2k+n+1-s)/2)} P_k $$ 
is a Fourier multiplier on $ L^p(\H^n)$. In view of the known multiplier theorems (\cite{MS}, \cite[Theorem 2.6.1]{STH}) this amounts to check that the function (as a function of $k$)
\begin{equation*}
(2k+n)^{-s} \frac{\Gamma\big(\frac{2k+n+1+s}{2}\big)}{\Gamma\big(\frac{2k+n+1-s}{2}\big)}
\end{equation*}
and its $j$th-derivatives up to order $(n+1)$ are bounded by $C_j k^{-j}$ for $j=0,1,\ldots, (n+1)$, which is true in view of the known asymptotics for the ratio of gamma functions (see for instance \cite{OlMax})
$$
\frac{\Gamma(z+a)}{\Gamma(z+b)}\sim z^{a-b}\quad \text{ as } z\to \infty, \quad |\operatorname{ph} z|<\pi,
$$
and the asymptotics of polygamma function, involved in the derivatives of the Gamma function.

 \section{Toolbox}
 \label{sec:technical}
In this section we collect and study several ingredients that will be used to prove the main theorems, namely the extension problem and the integral representation associated with $\mathcal{L}_s$,
  mapping properties for the square function and the nontangential square function, and a mean value theorem for subharmonic functions on $\H^n\times \R^+$.
  
 \subsection{The extension problem and a bilinear form associated with $\mathcal{L}_s$}
 \label{subsec:bilin}

Let
\begin{equation} 
\label{eq:varfisr}
\varphi_{s,\rho}(z,t)  = \big((\rho^2+|z|^2)^2 +16 t^2\big)^{-\frac{n+1+s}{2}},
\end{equation} 
which is integrable on $\H^n$ for all $s>0$, and 
\begin{equation}
\label{eq:C1}
C(n,s)=\frac{4}{\pi^{n+1/2}}\frac{\Gamma(n+s)\Gamma\big(\frac{n+1+s}{2}\big)}{\Gamma(s)\Gamma(\frac{n+s}{2})}.
\end{equation}
 The following theorem, which provides a solution to the extension problem \eqref{eq:epH}, realizes $\mathcal{L}_sf$ as the Dirichlet-to-Neumann map associated to the extension problem, and shows a pointwise representation for $\mathcal{L}_s$, can be found in \cite[Theorem 1.2]{RTimrn} (actually, here we are stating a reduced version of the result therein), see also \cite{RT}. 

\begin{theorem} \cite[Theorem 1.2]{RTimrn}
\label{thm:sol}
Let $s>0$. Let $f\in L^p(\H^n)$, $1\le p<\infty$. Then, as $\rho\to0^+$,
\begin{equation}
\label{eq:convergfT}
w=C(n,s)\rho^{2s}f\ast \varphi_{s,\rho} \rightarrow  f \qquad \text{ in } L^p(\H^n),
\end{equation}
where $\varphi_{s,\rho}$ is defined in \eqref{eq:varfisr} and $C(n,s)$ is the one in \eqref{eq:C1}.
 If we further assume that $\mathcal{L}_sf\in L^p(\H^n)$ then
\begin{equation*} 
-\lim_{\rho\to 0^+}\rho^{1-2s} \partial_\rho (w(z,t,\rho)) = 2^{1-2s}\frac{\Gamma(1-s)}{\Gamma(s)} \mathcal{L}_s f(z,t).
\end{equation*}
Moreover, when $0<s<1/2$, we also have the pointwise representation
\begin{equation}
\label{eq:ir}
\mathcal{L}_sf(x)=b(n,s)\int_{\H^n}\frac{f(x)-f(y)}{|xy^{-1}|^{Q+2s}}\,dy
\end{equation}
for all $f\in C^1(\H^n)$ such that $X_jf, Y_jf, \partial_tf\in L^{\infty}(\H^n)$, $j=1,\ldots, n$, where
\begin{equation}
\label{eq:bns}
b(n,s):=\frac{4^{1+s} }{\pi^{n+1/2}}\frac{\Gamma(n+s)\Gamma\big(\frac{n+1+s}{2}\big)}{\Gamma\big(\frac{n+s}{2}\big)|\Gamma(-s)|}.
\end{equation}
\end{theorem}

In view of Theorem \ref{thm:sol}, the function $C(n,s)\varphi_{s,\rho}$ is understood as a generalized conformal Poisson kernel, which is a solution to a generalized conformal harmonic extension. Observe that, when $s=1/2$ in \eqref{eq:epH}, we are reduced to the (conformal) harmonic extension
\begin{equation*}
\big(\partial_{\rho\rho}+\frac14\rho^2\partial_{tt}-\mathcal{L}\big)U(z,t,\rho)=0  \qquad \lim_{\rho\to 0}U(z,t,\rho)=f(z,t)\quad  \text{ in } \H^n\times \R^+,
\end{equation*}
so that the function $
U=C(n,1/2)f\ast \varphi_{1/2,\rho}$ in \eqref{eq:convergfT} is the Poisson semigroup. We will denote it as $U(z,t,\rho)=e^{-\rho\mathcal{L}_{1/2}}f(z,t)$. 

On the other hand, when $s=1/2$, the extension problem \eqref{eq:epHnoc} takes the form
\begin{equation}
\label{eq:he}
\big(\partial_{\rho}^2-\mathcal{L}\big)U=0 \quad \text{ in } \H^n\times \R^+, \qquad
U(z,t,0)=u(z,t), \qquad\text{ in } \H^n.
\end{equation}
We will denote $U(x,\rho):=e^{-\rho\mathcal{L}^{1/2}}u(x)$. There is not an explicit expression for the solution of this problem, however the subordination formula
$$
e^{-\rho\mathcal{L}^{1/2}}=\rho\int_0^{\infty}(4\pi w)^{-\frac12}e^{-\frac{\rho^2}{4w}}e^{-w\mathcal{L}}\,dw
$$
allows to write the non-conformal Poisson kernel $P_{\rho}(x)$ as 
$$
P_\rho(x)=\rho\int_0^{\infty}(4\pi w)^{-\frac12}e^{-\frac{\rho^2}{4w}}q_w(x)\,dw
$$
and the known sharp estimates for the heat kernel $q_w(z,t)$ (see e.g. \cite[Proposition 2.8.2]{STU}) yield the following sharp estimates for the Poisson kernel (see also \cite[Theorem 6.11 (i)]{ADY})
\begin{equation}
\label{eq:Pestim}
P_\rho(x)\le C_n\frac{\rho}{(\rho^2+|x|^2)^{\frac{Q+1}{2}}}.
\end{equation}
The latter will be crucial to prove mapping properties of the square function operators given in the next subsection.

 \subsection{The square functions $g$, $g_{\lambda}^*$}

Recall that we are letting  
$$
\nabla =(X_1, \ldots, X_{n}, Y_1,\ldots, Y_{n}, \partial_{\rho}).
$$ 
We define the \textit{square $g$-function} by
$$
g(u)(x)=\Big(\int_0^{\infty}|\nabla U(x,\rho)|^2 \rho\,d\rho\Big)^{1/2}
$$
and the \textit{Littlewood nontangential square function} $g_{\lambda}^*$ as
\begin{equation}
\label{eq:gstar2}
g_{\lambda}^*(u)(x):=\Big(\int_0^{\infty}\int_{\H^n}\Big(\frac{\rho}{\rho+|y|}\Big)^{\lambda Q}\rho^{1-Q}|\nabla U(xy^{-1},\rho)|^2\,dy\,d\rho\Big)^{1/2},
\end{equation}
where $U(x,\rho)$ is the non-conformal harmonic extension of $u(x)$ in the upper half space in \eqref{eq:he}. 
In \cite[Chapter 2.6]{STH}, Thangavelu defined $g$ and $g^*$ functions in terms of the heat semigroup and proved $L^p$ mapping properties for these operators. 

For $g$, the following basic result can be proven as in the Euclidean case, see \cite[Chapter 4, \S 1]{St}. 
We can equip $\H^n$ with a metric induced by the Koranyi norm which makes it a space of homogeneous type. On such spaces there is a well defined notion of dyadic cubes and grids with properties similar to their counterparts in the Euclidean setting. 
Given $1<p<\infty$, by $A_p$ we will denote the Muckenhoupt class of weights in $\H^n$, namely all nonnegative functions $w\in L^1_{\operatorname{loc}}(\H^n)$ such that
$$
\sup_{Q}\Big(\frac{1}{|Q|}\int_Qw(x)\,dx\Big)\Big(\frac{1}{|Q|}\int_Qw(x)^{-p'/p}\,dx\Big)^{p/p'}<\infty
$$ 
where the supremum is taken over all cubes $Q\in \H^n$.
 \begin{theorem}
  \label{thm:g}
  Let $n\ge1$ and $w\in A_p$. For any  $u\in L^p(w):= L^p(\H^n,w)$ we have, for $1<p<\infty$,
  $$
 c_p\|u\|_{L^p(w)}\le \|g(u)\|_{L^p(w)}\le C_p \|u\|_{L^p(w)}.
  $$
  \end{theorem}
  
  \begin{proof}
  The proof follows classical arguments, we point out the main steps. Observe that $|\nabla U(x,\rho)|^2=|\partial_{\rho}U|^2+|\nabla_x U(x,\rho)|^2$, where $|\nabla_x U(x,\rho)|^2=\sum_{j=1}^n(|X_jU|^2+|Y_jU|^2)$. It will be appropriate to introduce the following two partial $g$-functions, namely
  \begin{equation*}
g_1(u)(x)=\Big(\int_0^{\infty}|\partial_{\rho}U(x,\rho)|^2 \rho\,d\rho\Big)^{1/2}, \qquad g_x(u)(x)=\Big(\int_0^{\infty}|\nabla_x U(x,\rho)|^2 \rho\,d\rho\Big)^{1/2}.
  \end{equation*}
  Note that $g^2=g_1^2+g_x^2$.
  
  Let us focus on the $L^2$ estimate for $g_1$. Applying Plancherel theorem for the Fourier transform on $\H^n$, we get (call $U(x,\rho)=:U_{\rho}(x)$)
  $$
  \|g_1(u)\|_{L^2}^2=\int_0^{\infty}\int_{\H^n}|\partial_{\rho}U(x,\rho)|^2 \rho\,dx\,d\rho=\frac{2^{n-1}}{\pi^{n+1}}\int_0^{\infty}\int_{-\infty}^{\infty}\|\widehat{(\partial_{\rho}U_{\rho})}(\lambda)\|^2_{\operatorname{HS}}|\lambda|^n \,d\lambda\, \rho\,d\rho.
  $$
  We also have
  $$
  \widehat{(\partial_{\rho}U_{\rho})}(\lambda)=-\widehat{u}(\lambda)H(\lambda)^{1/2}e^{-\rho H(\lambda)^{1/2}}
  $$
  and so the squared Hilbert-Schmidt norm is given by the sum
  $$
  \sum_{\alpha\in \mathbb{N}^n}(2|\alpha|+n)|\lambda|e^{-2\rho ((2|\alpha|+n)|\lambda|)^{1/2}}\|\widehat{u}(\lambda)\Phi_{\alpha}^{\lambda}\|_{L^2}^2.
  $$
  Integrating the above with respect to $\rho \,d\rho$, we get
  \begin{equation}
  \label{eq:g1L2}
   \|g_1(u)\|_{L^2}^2=\frac14\frac{2^{n-1}}{\pi^{n+1}}\int_{-\infty}^{\infty}\|\widehat{u}(\lambda)\|_{\operatorname{HS}}^2|\lambda|^n\,d\lambda=\frac14\|u\|_{L^2}^2.
  \end{equation}
  
Now, for $p\neq 2$, the converse inequality $\|u\|_{L^p(w)}\le c\|g(u)\|_{L^p(w)}$ can be derived with a polarization argument from the $L^2$ identity \eqref{eq:g1L2} involving the weight and its dual, as in \cite[Section 6]{CR} and the fact that $g_1(x)\le g(x)$ implies $\|g_1\|_{L^p(w)}\le \|g\|_{L^p(w)}$.

The inequality $\|g(u)\|_{L^p(w)}\le C\|u\|_{L^p(w)}$ for $1<p<\infty$ follows as in \cite[Chapter 4, \S 1]{St} by using the estimates for the Poisson kernel \eqref{eq:Pestim} and the theory of vector-valued operators in spaces of homogeneous type \cite{GLY, Ruiz-Torrea}. 
  \end{proof}
We will also need to prove estimates for $g_{\lambda}^*$.
  \begin{theorem}
  \label{thm:gstar}
  Let $n\ge1$ and $\lambda>1$. For any  $u\in L^p(\H^n)$ we have, for $1<p<\infty$,
  $$
  \|g_{\lambda}^{*}(u)\|_{L^p}\lesssim \|u\|_{L^p}.
  $$
for $\lambda>\max\big\{1,\frac2p\big\}$.
  \end{theorem}
  
  \begin{proof}  The proof of Theorem \ref{thm:gstar} also follows the lines of the corresponding Euclidean result, see \cite[Chapter 4, \S 2]{St}. We detail the pertinent ingredients.
 
  We start with the case $p\ge2$, where only the hypothesis $\lambda>1$ is relevant. We can write
  $$
  \|g_{\lambda}^*(u)\|_{L^p}^2=\sup_{\|\psi\|_{L^{(p/2)'}}\le 1}\int_0^{\infty}\int_{\H^n}\rho|\nabla U(y,\rho)|^2\Big(\int_{\H^n}\Big(\frac{\rho}{\rho +|x^{-1}y|}\Big)^{\lambda Q}\rho^{-Q}\psi(x)\,dx\Big)\,dy\,d\rho.
  $$
  On the other hand
  $$
 \sup_{\rho>0} \int_{\H^n}\Big(\frac{\rho}{\rho  +|x^{-1}y|}\Big)^{\lambda Q}\rho^{-Q}\psi(x)\,dx\le C\sup_{\rho>0}(\psi\ast \varphi_{\rho})(y),
  $$
  where  $\varphi_{\rho}(x)=\rho^{-Q}\varphi(x/\rho)$ with $\varphi(x)=(1+|x|)^{-\lambda Q}$. It can be proved, analogously as in \cite[Theorem 2, \S2.2 of Chapter III]{St}, that
 $$
  \sup_{\rho>0}(\psi\ast \varphi_{\rho})(y)\le C M\psi(y)
  $$
where, for $f\in L_{\operatorname{loc}}^1(\H^n)$, $M$ is the centered Hardy-Littlewood maximal function given by 
$$
Mf(z,t)=\sup_{r>0}\frac{1}{|B((z,t),r)|}\int_{B((z,t),r)}|f(z',t')|\,dz', \quad  (z,t)\in \H^n.
$$
Here $B((z,t),r)$ denotes the open ball with  center $(z,t)$ and radius $r$ induced by the Kor\'anyi norm $d$ and, for a measurable set $A$, we denote the volume by $|A|$. 
 Observe that the Hardy-Littlewood maximal operator is bounded in $L^p(\H^n)$ for $1<p<\infty$. This follows from general results on a space of homogeneous type in the sense of Coifman and Weiss \cite{CW}. 
  Therefore, 
$$
  \|g_{\lambda}^*(u)\|_{L^p}^2\le C_{\lambda}\int_{\H^n}|g(u)(y)|^2 M\psi(y)\,dy\le C_{\lambda} \|g(u)\|_{L^p}^2\|M\psi\|_{L^{(p/2)'}}\le C_{\lambda}\|f\|_{L^p}^2.
$$
 Here we have used Theorem \ref{thm:g} and the boundedness of the maximal function. 
 
Let us move to the case $p<2$.  Let $\mu\ge 1$ and write the following variant of the maximal function
$$
M_{\mu}f(z,t)=\Big(\sup_{r>0}\frac{1}{|B((z,t),r)|}\int_{B((z,t),r)}|f(z',t')|^{\mu}\,dz'\Big)^{1/\mu}, \quad  (z,t)\in \H^n
$$
for which the following holds
$$
\|M_{\mu}f\|_{L^p}\le C_{p,\mu}\|f\|_{L^p}, \quad p>\mu.
$$
We also have the estimate
\begin{equation*}
|U(x^{-1}y,\rho)|\le C_{\mu}\Big(1+\frac{|y|}{\rho}\Big)^{Q/\mu}M_{\mu}f(x).
\end{equation*}
The above follows analogously as in \cite[Chapter IV, Lemma 4]{St} with the help of the bound for the Poisson kernel in \eqref{eq:Pestim}. With this, proceeding as in \cite[\S 2.1]{St}, we have
$$
(g_{\lambda}^*(u)(x))^2=\frac{1}{p(p-1)}\int_0^{\infty}\int_{\H^n}\Big(\frac{\rho}{\rho+|y|}\Big)^{\lambda Q}\rho^{1-Q}U^{2-p}|\mathcal{L}U^p|\,dy\,d\rho\le C_{\mu}^{2-p}(M_{\mu}u(x))^{2-p}I^*(x),
$$
with
$$
I^*(x)=\int_0^{\infty}\int_{\H^n}\Big(\frac{\rho}{\rho+|y|}\Big)^{\lambda' Q}\rho^{1-Q}U^{2-p}\mathcal{L}U^p(x^{-1}y,\rho)\,dy\,d\rho.
$$
Observe that 
\begin{align*}
\int_{\H^n}I^*(x)\,dx&=\int_0^{\infty}\int_{\H^n}\int_{\H^n}\Big(\frac{\rho}{\rho+|xy^{-1}|}\Big)^{\lambda' Q}\rho^{1-Q}\mathcal{L}U^p(y,\rho)\,dx\,dy\,d\rho\\
&=C_{\lambda'}\int_0^{\infty}\int_{\H^n}\rho\mathcal{L}U^p(y,\rho)\,dy\,d\rho
\end{align*}
where in the last step we used that, for $\lambda'>1$,
$$
\rho^{-Q}\int_{\H^n}\Big(\frac{\rho}{\rho+|x|}\Big)^{\lambda' Q}\,dx=\int_{\H^n}\Big(\frac{1}{1+|x|}\Big)^{\lambda' Q}\,dx=C_{\lambda'}<\infty.
$$
It is easy to check that an analogous to \cite[Lemma 2]{St} also holds in our context, namely
$$
\int_0^{\infty}\int_{\H^n}\rho\mathcal{L}U^p(y,\rho)\,dy\,d\rho)=\int_{\H^n}U^p(y,0)\,dy.
$$
Gathering all the ingredients above, we infer that
\begin{equation}
\label{eq:Is}
\int_{\H^n}I^*(x)\,dx=C_{\lambda'}\|u\|_{L^p}^p.
\end{equation}
Finally, by H\"older's inequality with exponents $r$ and $r'$, $1<r<2$ (which is possible by the fact that $\big(\frac{2-p}{2}\big)pr'=p$ and $rp/2=1$ if $r=2/p$)
\begin{align*}
\int_{\H^n}(g_{\lambda}^*(u)(x))^p\,dx&\le C\int_{\H^n}(M_{\mu}u(x))^{p(2-p)/2}(I^*(x))^{p/2}\,dx\\
&\le C\Big(\int_{\H^n}(M_{\mu}u(x))^{p}\,dx\Big)^{1/r'}\Big(\int_{\H^n}I^*(x)\,dx\Big)^{1/r}.
\end{align*}
The conclusion follows in view of the boundedness of the maximal operator $M_{\mu}$ and \eqref{eq:Is}.
   \end{proof}

With the corresponding modifications in the proof, a weighted version of Theorem \ref{thm:gstar} can be also obtained, (or by adapting the proof in \cite[Corollary, p. 110]{MW} to an space of homogeneous type, taking into account the estimates for the Poisson kernel \ref{eq:Pestim}), using the weighted boundedness for the Hardy-Littlewood maximal operator on spaces of homogeneous type (e.g. \cite{AM}).

  \begin{theorem}
  \label{thm:gstarw}
  Let $n\ge1$ and $\lambda>1$. For any  $u\in L^p(\H^n)$ we have, for $1<p<\infty$, and $w\in A_{\min\{p,\frac{p\lambda}{2}\}}$,
  $$
  \|g_{\lambda}^{*}(u)\|_{L^p(w)}\lesssim \|u\|_{L^p(w)}.
  $$
for $\lambda>\max\big\{1,\frac2p\big\}$.
  \end{theorem}

\subsection{A mean value theorem for subharmonic functions on $\H^n\times \R^+$}
\label{sub:mvt}

 Let us write $E:=-\mathcal{L}+\partial_{\rho\rho}=\nabla\cdot \nabla$ and
$$
|\nabla u|^2=\sum_{j=1}^n\big((X_ju)^2+(Y_ju)^2\big)+\partial_{\rho\rho}u,
$$
where $X_j, Y_j$, $j=1,\ldots, n$ are given in \eqref{vfields}. It turns out that $E$ is homogeneous of degree $2$, hypoelliptic\footnote{Recall that a differential operator $D$ is hypoelliptic if the solutions of the equation $Df=g$ with $g\in C^{\infty}$, are also $C^{\infty}$. In our case, by a theorem of H\"ormander \cite{Ho}, since $X_j, Y_j$, $j=1,\ldots, n$ are vector fields with the property that their commutators up to a certain order span the tangent space at every point, then $\sum_{j=1}^n\big((X_ju)^2+(Y_ju)^2\big)$ is hypoelliptic, and hence $E$ is.} and formally self-adjoint. These facts imply that it possesses a fundamental solution $\Gamma$ which is $C^{\infty}$ off the diagonal in $ \H^n\times \R^+\times  \H^n\times \R^+$, see \cite{F1,SC}.

Following for instance \cite{BL}, a function $h$ will be called $E$-\textit{harmonic} in an open set $\Omega\subset \R_+\times \H^n$ if $h:\Omega\to \R$ is smooth and $Eu=0$ in $\Omega$. An upper semicontinuous function $u:\Omega\to \R$ will be said $E$-\textit{subharmonic} in $\Omega$ if 
\begin{itemize}
\item[(i)] the set $\Omega(u):=\{x\in \Omega: u(x)>-\infty\}$ contains at least one point of every (connected) component of $\Omega$, and
\item[(ii)] for every bounded set $V\subset \overline{V}\subset \Omega$ and for every $E$-harmonic function $h\in C^2(V,\R)\cap C(\overline{V},\R)$ such that $u\le h$ on $\partial V$, one has $u\le h$ in $V$. 
\end{itemize}
A subharmonic function $u$ satisfies that $Eu\ge 0$ on $\Omega$.

 Starting from a result of \cite{CGL} related to general hypoelliptic operators sum of squares of vector fields, representation formulas for subharmonic functions on Carnot groups are proved in \cite{BL1}, see also \cite{BL} (a full discussion can be found in \cite[Chapter 5]{BLU}). We rewrite the results specified to our context.

Recall that $d(z,t)=|(z,t)| = (|z|^4+16 t^2)^{1/4}$. Let us denote $\widetilde{d}((z,t),\rho)=(\rho^4+|z|^4+16 t^2)^{1/4}$. 
\begin{equation}
\label{eq:ballsph}
B_r=\{((z,t), \rho)\in \H^n\times \R^+: \widetilde{d}((z,t),\rho)<r\}, 
\end{equation}
and 
\begin{equation}
\label{eq:ballsph2}
 \partial B_r=\{((z,t), \rho)\in \H^n\times \R^+:\widetilde{d}((z,t),\rho)=r\},
\end{equation}
and call these sets, respectively, the \textit{extended Heisenberg ball} and \textit{extended sphere} centered at the origin with radius $r$. Balls and spheres centered at points other than the origin are defined by left-translation and the usual Euclidean distance. We let $d((z,t),(z',t'))=d((z',t')^{-1}(z,t))$ denote the distance between $(z,t)$ and $(z',t')$. Then the ball $B_r((z',t'),\rho')=B(((z',t'),\rho'),r)$ and the sphere $\partial B_r((z',t'),\rho')$ centered at $(z',t',\rho')$ with radius $r$ are obtained by replacing $ (\rho^4+|z|^4+16 t^2)^{1/4}$ in \eqref{eq:ballsph} and \eqref{eq:ballsph2} with $\widetilde{d}(((z,t),\rho),((z',t'),\rho'))=\big(d((z,t),(z',t'))^4+|\rho-\rho'|^4\big)^{1/4}$. For simplicity, below we will denote $x=((z,t),\rho)$, $y=((z',t'),\rho')$ and $x^{-1}y=((z,t)^{-1}(z',t'), \rho'-\rho)$.

\begin{theorem}
\label{thm:formulas}
Let $u\in C^{\infty}(\H^n\times \R^+)$ be a $E$-subharmonic function in an open subset $\Omega\subset \H^n\times \R^+$. Then for every $x=((z,t),\rho)\in\H^n\times \R^+$ and $r>0$ such that $\overline{B_r}(x)\subset \Omega\subset \H^n\times \R^+$ we have
\begin{equation}
\label{eq:mvtB}
u(x)\le \frac{C_n}{r^{Q+1}}\int_{B_r(x)}K(x^{-1}y)u(y)\,dy,  \qquad K:=|\nabla \widetilde{d}|^2.
\end{equation}
\end{theorem}
It can be checked that $|\nabla \widetilde{d}((z,t),\rho)|^2\le 1$ in $\H^n\times \R^+$ so \eqref{eq:mvtB} in Theorem \ref{thm:formulas} yields that, for a $E$-subharmonic function $u$ in an open subset $\Omega\subset \H^n\times \R^+$,
\begin{equation}
\label{eq:mean}
u(x)\le \frac{C_n}{r^{Q+1}}\int_{B_r(x)}K(x^{-1}y)u(y)\,dy\le \frac{C_n}{r^{Q+1}}\int_{B_r(x)}u(y)\,dy.
\end{equation}

 \section{Proof of Theorem \ref{thm:pointwise}}
 \label{sec:proofTh2}
   
   The proof follows the argument sketched by Stein in \cite{St}, and explicitly written in detail in \cite{D}. Without loss of generality, we will prove the inequality at $x=0$. Let $u\in \mathcal{S}(\H^n)$ and recall that we denote by $U(x,\rho):=e^{-\rho\mathcal{L}^{1/2}}u(x)$ the Poisson semigroup associated with the non-conformal harmonic extension \eqref{eq:he}. In $\H^n\times \R^+$, let $\Gamma$ be the path joining the points $(0,0)$ and $(y,0)$, consisting of the line segments joining $(0,0)$ with $(0,|y|)$, $(0,|y|)$ with $(y,|y|)$, and $(y,|y|)$ with $(y,0)$, where $y\in \H^n$. By Stokes' theorem we have
 $$
 \int_{\Gamma}\nabla U \,d\lambda= U(y,0)-U(0,0)= u(y)-u(0), 
 $$
 which implies
 $$
 |u(y)-u(0)|\le \int_0^{|y|}\big(|\nabla U(y,\lambda)|+|\nabla U(0,\lambda)|+|\nabla U(\lambda\widehat{y},|y|)|\big)\,d \lambda,\qquad \widehat{y}=\frac{y}{|y|}.
 $$
 Let us denote by $F(x,\rho):=e^{-\rho \mathcal{L}^{1/2}}(\mathcal{L}^{s/2}u)$ the non-conformal harmonic extension of $\mathcal{L}^{s/2}u$.
 
 We have the following fundamental identity.

\begin{lemma}
\label{lem:identity}
Let $u\in \mathcal{S}(\H^n)$. Then, for $0<s<n+1$,
$$
U(x,\rho)=\frac{1}{\Gamma(s)}\int_0^{\infty}F(x,\rho+\mu)\mu^{s-1}\,d\mu.
$$
\end{lemma}
  \begin{proof}
By using the semigroup property of $e^{-\rho\mathcal{L}^{1/2}}u(x)$ and the fact that $\mathcal{L}^{-s/2}\mathcal{L}^{s/2}=\operatorname{Id}$, we get
  \begin{align*}
  \int_0^{\infty}F(x,\rho+\mu)\mu^{s-1}\,d\mu&=\int_0^{\infty}e^{-(\rho+\mu)\mathcal{L}^{1/2}}(\mathcal{L}^{s/2}u)(x)\mu^{s-1}\,d\mu\\
  &=e^{-\rho\mathcal{L}^{1/2}}\int_0^{\infty}e^{-\mu\mathcal{L}^{1/2}}(\mathcal{L}^{s/2}u)(x)\mu^{s-1}\,d\mu\\
  &=\Gamma(s)e^{-\rho\mathcal{L}^{1/2}}\mathcal{L}^{-s/2}(\mathcal{L}^{s/2})u(x)\\
  &=\Gamma(s)e^{-\rho\mathcal{L}^{1/2}}u(x)=\Gamma(s)U(x,\rho),
  \end{align*}
  as desired.
  \end{proof}
  
  It is important to remark that Lemma \ref{lem:identity} is strongly based on the fact that, for $0<s<n+1$, the Riesz potentials $\mathcal{L}^{-s/2}$ are defined via function calculus in terms of the (non-conformal) Poisson semigroup associated with $\mathcal{L}$ as
$$
\mathcal{L}^{-s/2}f(z,t)=\frac{1}{\Gamma(s)}\int_0^{\infty}e^{-\mu\mathcal{L}^{1/2}}f(z,t)\frac{d\mu}{\mu^{1-s}}.
$$
There is no such an analogue representation for the Riesz potentials $\mathcal{L}_{-s/2}$ in terms of the solution $e^{-\rho\mathcal{L}_{1/2}}$ of the conformally invariant harmonic extension (see the interesting discussion in \cite[Section 1, (1.3) and (1.11)]{GT}) which is refraining us to use a $g_{\lambda}^*$-function defined with $e^{-\rho\mathcal{L}_{1/2}}$.

Lemma \ref{lem:identity} and a change of variables yield
\begin{align*}
|u(y)-u(0)|
&\le \int_0^{|y|}\Big(\int_{\lambda}^{\infty}\big(|\nabla F(y,\mu)|+|\nabla F(0,\mu)|+|\nabla F(\lambda\widehat{y},\mu+|y|-\lambda)| \big)(\mu-\lambda)^{s-1}\,d\mu\Big)\,d\lambda\\
&=: I+II+III+IV,
\end{align*}
where 
\begin{align*}
I&:=\int_0^{|y|}\int_{\lambda}^{|y|}|\nabla F(y,\mu)|(\mu-\lambda)^{s-1}\,d\mu\,d\lambda, \qquad II:=\int_0^{|y|}\int_{\lambda}^{|y|}|\nabla F(0,\mu)|(\mu-\lambda)^{s-1}\,d\mu\,d\lambda,\\
III&:=\int_0^{|y|}\int_{\lambda}^{|y|}|\nabla F(\lambda\widehat{y},\mu+|y|-\lambda)|(\mu-\lambda)^{s-1}\,d\mu\,d\lambda, \\
 IV&:=\int_0^{|y|}\int_{|y|}^{\infty}\big(|\nabla F(y,\mu)|+|\nabla F(0,\mu)|+|\nabla F(\lambda\widehat{y},\mu+|y|-\lambda)| \big)(\mu-\lambda)^{s-1}\,d\mu\big)\,d\lambda.
\end{align*}

\textbf{Estimate of $IV$}. Observe that, for $A>0$ and $s<A<1$,
\begin{equation}
\label{eq:IV}
IV\lesssim \sup_{\substack{ |x|\le |y|\le \rho\\ \rho>|y|}}|\nabla F(x,\rho)|\rho^A\int_0^{|y|}\int_{|y|}^{\infty}\mu^{-A}(\mu-\lambda)^{s-1}\,d\mu\,d\lambda\lesssim  \sup_{\substack{ |x|\le |y|\le \rho\\ \rho>|y|}}|\nabla F(x,\rho)|\rho^A|y|^{1+s-A}. 
\end{equation}
Since $|\nabla F|^2$ is subharmonic, namely $(\partial_{\rho\rho}-\mathcal{L})|\nabla F|^2>0$, and since 
$$
D(x,\rho):=\big\{(\xi,\tau)\in \H^n\times \R^+: |\xi-x|<\frac{\rho}{2}, |\tau-\rho|^2<\frac{\rho}{2}\big\}
$$ 
has equivalent measure to $B_{\lambda/2}(x,\rho)$ we get, from $\eqref{eq:IV}$ and \eqref{eq:mean}, that
  $$
 (IV)^2\lesssim  \sup_{ |x|\le |y|\le \rho}\rho^{2A}|y|^{2+2s-2A}\int_D|\nabla F(\xi,\tau)|^2\,d\xi\,d\tau \frac{1}{\rho^{1+Q}}.
 $$

Let 
$$
C:= \{(\xi,\tau)\in \H^n\times \R^+: |\xi|\le 3\tau\}.
$$ 
\begin{figure}[h]
\includegraphics[scale=0.6]{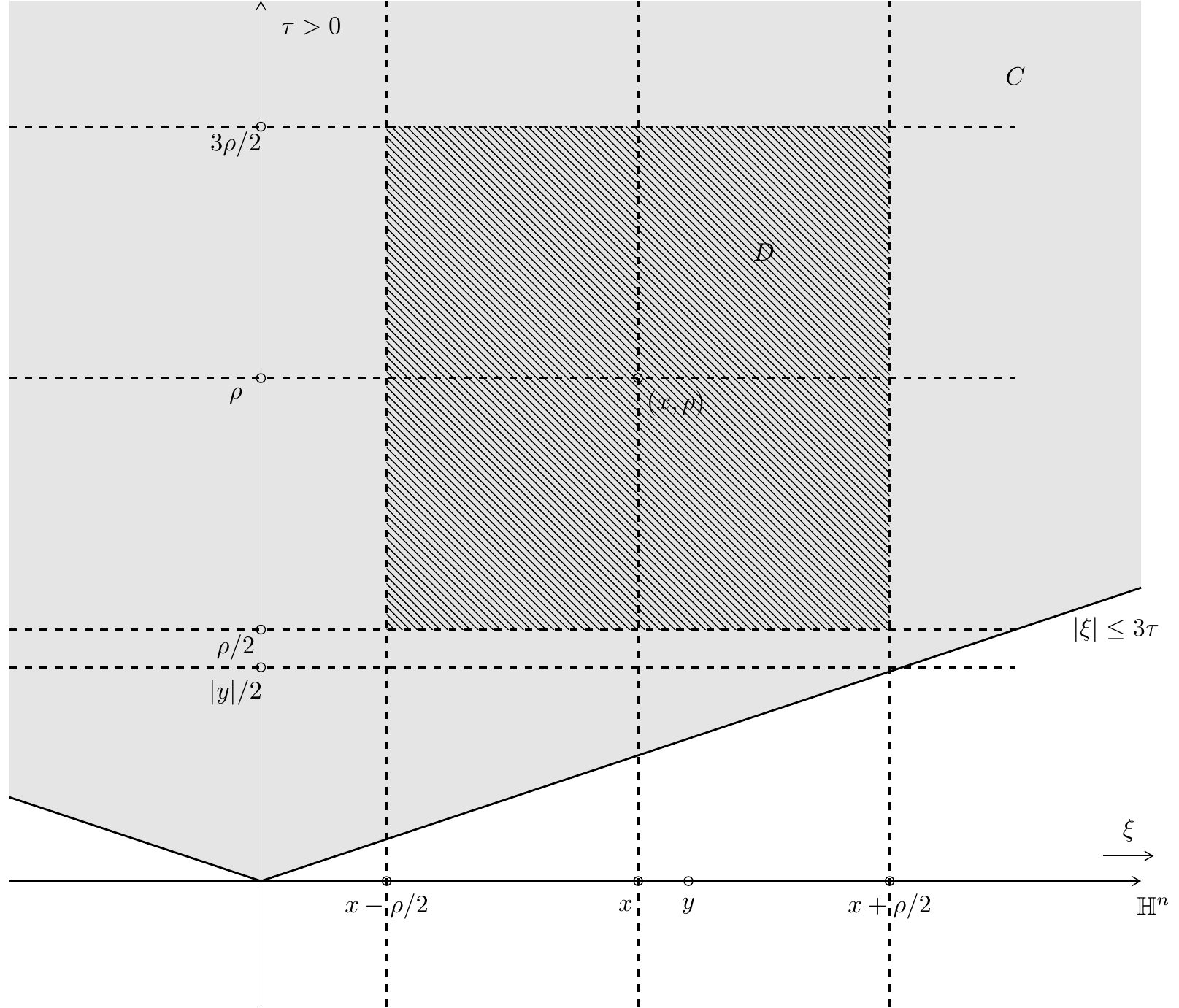}
\caption{The sets $D$ and $C$.}
\label{figure}
\end{figure}
Notice that $D\subset \{(\xi,\tau)\in C: \tau\ge |y|/2\}$
and in $D$ we have that $\rho/2\le \tau\le 3/2 \rho$, see Figure \ref{figure}, thus
 $$
 (IV)^2\lesssim  |y|^{2+2s-2A}\int_{C,\tau\ge |y|/2}|\nabla F(\xi,\tau)|^2\tau^{2A-Q-1}\,d\xi\,d\tau \frac{1}{\rho^{1+Q}}.
 $$
 We divide by $|y|^{Q+2s}$ and integrate in $y$, to obtain
 \begin{equation}
 \label{eq:IVf}
 \int_{\H^n}\frac{(IV)^2}{|y|^{Q+2s}}\,dy\lesssim \int_{C}\Big(\int_{|y|\le 2\tau} |y|^{2-Q-2A}\,dy\Big)|\nabla F(\xi,\tau)|^2\tau^{2A-Q-1}\,d\xi\,d\tau\lesssim  \int_{C}|\nabla F(\xi,\tau)|^2\tau^{1-Q}\,d\xi\,d\tau.
 \end{equation}

\textbf{Estimate of $III$}. In this case, $0\le\lambda\le \mu\le |y|$, so that $\mu+|y|-\lambda\ge \lambda+|y|-\lambda=|y|\ge \lambda=|\lambda \widehat{y}|$, and then
\begin{align*}
III&\lesssim  \sup_{ |x|\le |y|\le \rho}\rho^A|\nabla F(x,\rho)|\int_0^{|y|}\int_{\lambda}^{|y|} (|y|+\mu-\lambda)^{-A}(\mu-\lambda)^{s-1}\,d\xi\,d\lambda\\
&\lesssim \sup_{ |x|\le |y|\le \rho}\rho^A|\nabla F(x,\rho)||y|^{-A+1+s}.
 \end{align*}
 Hence, reasoning as in the case of $IV$, we conclude that $III$ satisfies \eqref{eq:IVf}, namely
 \begin{equation}
 \label{eq:IIIf}
 \int_{\H^n}\frac{(IV)^2}{|y|^{Q+2s}}\,dy\lesssim  \int_{C}|\nabla F(\xi,\tau)|^2\tau^{1-Q}\,d\xi\,d\tau.
 \end{equation}
 
\textbf{Estimate of $II$}. Let $0<\varepsilon<2s$. By Fubini and Cauchy-Schwartz, we have
 \begin{align*}
 II&=\int_0^{|y|}\Big(\int_{0}^{\mu}|\nabla F(0,\mu)|(\mu-\lambda)^{s-1}\,d\lambda\Big)\,d\mu\\
 &\simeq \int_0^{|y|}|\nabla F(0,\mu)|\mu^s\,d\mu\\
 &\le \Big(\int_0^{|y|}|\nabla F(0,\mu)|^2\mu^{1+2s-\varepsilon}\,d\mu\Big)^{1/2}\Big(\int_0^{|y|}\mu^{\varepsilon-1}\,d\mu\Big)^{1/2}\\
 &\lesssim \Big(\int_0^{|y|}|\nabla F(0,\mu)|^2\mu^{1+2s-\varepsilon}\,d\mu\Big)^{1/2}|y|^{\varepsilon/2}.
 \end{align*}
 From here, 
 $$
 (II)^2\lesssim \int_0^{|y|}|\nabla F(0,\mu)|^2\mu^{1+2s-\varepsilon}\,d\mu|y|^{\varepsilon},
 $$
 thus, applying Fubini 
 \begin{align*}
  \int_{\H^n}\frac{(II)^2}{|y|^{Q+2s}}\,dy&\lesssim \int_{\H^n} |y|^{\varepsilon-Q-2s}\Big(\int_0^{|y|}|\nabla F(0,\mu)|^2\mu^{1+2s-\varepsilon}\,d\mu\Big)\,dy\\
  &=\int_0^{\infty}\Big( \int_{|y|\ge \mu} |y|^{\varepsilon-Q-2s}\,dy\Big)|\nabla F(0,\mu)|^2\mu^{1+2s-\varepsilon}\,d\mu\\
  &\lesssim \int_0^{\infty}|\nabla F(0,\mu)|^2\mu\,d\mu.
 \end{align*}
Then by \eqref{eq:mean} we obtain
$$
 \int_{\H^n}\frac{(II)^2}{|y|^{Q+2s}}\,dy\lesssim \int_0^{\infty}\mu\mu^{-1-Q}\int_{E}|\nabla F(\xi,\tau)|^2\, d\xi\, d\tau\,d\mu
$$
with $E:=\{(\xi,\tau): |\xi|\le \mu/2, |\tau-\mu|\le \mu/2\}$. Notice that, in $E$, $\mu/2\le \tau\le 3/2\mu$, so $2/3\tau\le \mu\le 2\tau$ and $E\subset C$, where $C$ is the cone defined above. By Fubini,
 \begin{equation}
 \label{eq:IIf}
  \int_{\H^n}\frac{(II)^2}{|y|^{Q+2s}}\,dy\lesssim \int_C|\nabla F(\xi,\tau)|^2\Big(\int_{2/3\tau}^{2\tau}\mu^{-Q}\,d\mu\Big)\,d\xi\,d\tau\lesssim  \int_C|\nabla F(\xi,\tau)|^2\tau^{1-Q}\,d\xi\,d\tau.
 \end{equation}

\textbf{Estimate of $I$}. Let $0<\varepsilon<2s$. By Fubini and Cauchy-Schwartz, we get
 \begin{align*}
 I=\int_0^{|y|}\Big(\int_{0}^{\mu}|\nabla F(y,\mu)|(\mu-\lambda)^{s-1}\,d\lambda\Big)\,d\mu&\simeq \int_0^{|y|}|\nabla F(y,\mu)|\mu^s\,d\mu\\
 &\lesssim \Big(\int_0^{|y|}|\nabla F(0,\mu)|^2\mu^{1+2s-\varepsilon}\,d\mu\Big)^{1/2}|y|^{\varepsilon/2}.
  \end{align*}
 Therefore, by using Fubini again
 \begin{align*}
  \int_{\H^n}\frac{(I)^2}{|y|^{Q+2s}}\,dy&\lesssim \int_{\H^n} |y|^{\varepsilon-Q-2s}\Big(\int_0^{|y|}|\nabla F(y,\mu)|^2\mu^{1+2s-\varepsilon}\,d\mu\Big)\,dy\\
  &=\int_0^{\infty}\Big(\int_{|y|\ge \mu}|y|^{\varepsilon-Q-2s}|\nabla F(y,\mu)|^2\,dy\Big)\mu^{1+2s-\varepsilon}\,d\mu.
   \end{align*}
Hence,
   \begin{equation}
   \label{eq:If}
 \int_{\H^n}\frac{(I)^2}{|y|^{Q+2s}}\,dy  \lesssim \int_{|y|\ge \mu}|\nabla F(y,\mu)|^2\frac{\mu^{1+2s-\varepsilon}}{|y|^{Q+2s-\varepsilon}}\,dy\,d\mu.
  \end{equation}

 Gathering \eqref{eq:IVf},  \eqref{eq:IIIf}, \eqref{eq:IIf}, and \eqref{eq:If}, we have that, for $\varepsilon\in (0,2s)$, 
  \begin{equation}
  \label{eq:todo}
  \int_{\H^n}\frac{|u(y)-u(0)|^2}{|y|^{Q+2s}}\,dy\lesssim \int_C|\nabla F(\xi,\tau)|^2\tau^{1-Q}\,d\xi\,d\tau+\int_{|y|\le \mu}|\nabla F(y,\mu)|^2\frac{\mu^{1+2s-\varepsilon}}{|y|^{Q+2s-\varepsilon}}\,dy\,d\mu. 
  \end{equation}
  
Recall the definition of $g_{\lambda}^*$ in \eqref{eq:gstar2}, then 
$$
g_{\lambda}^*(\mathcal{L}^{s/2}u)(0)=\int_0^{\infty}\int_{\H^n}\Big(\frac{\tau}{\tau+|\xi|}\Big)^{\lambda Q}\rho^{1-Q}|\nabla F(\xi,\tau)|^2\,d\rho.
$$
It is easy to see that the first integral in the right hand side of \eqref{eq:todo} can be estimated by $g_{\lambda}^*(\mathcal{L}^{s/2}u)(0)$ for any $\lambda$. For the second one, observe that the restriction $\mu\le |y|$ implies that $\frac{1}{|y|}\le \frac{2}{\mu+|y|}$, therefore 
$$
\frac{\mu^{1+2s-\varepsilon}}{|y|^{Q+2s-\varepsilon}}\lesssim \frac{\mu^{1+2s-\varepsilon}}{(|y|+\mu)^{Q+2s-\varepsilon}}.
$$
Then the second term is bounded by $g_{\lambda}^*(\mathcal{L}^{s/2}u)(0)$ with $\lambda=\frac{1}{Q}(Q+2s-\varepsilon)$. Since $\varepsilon$ is arbitrarily small, the proof is finished.

 \section{Proof of Theorem \ref{thm:main}}
 \label{sec:proofTh1}
 
 Let us call $x=(z,t)$. The pointwise representation \eqref{eq:ir} implies the identity
$$
\mathcal{L}_s(uv)-u\mathcal{L}_sv-v\mathcal{L}_su=b(n,s)T_s(u,v), \quad 0<s<1/2,
$$
where $b(n,s)$ is the constant in \eqref{eq:bns} and $T_s(u,v)$ is the bilinear form
$$
T_s(u,v)(x)=\int_{\H^n}\frac{[u(xy^{-1})-u(x)(v(xy^{-1})-v(x)]}{|y|^{Q+2s}}\,dy, \qquad x\in \H^n, \quad 0<s<1/2.
$$
Then, in order to prove Theorem \ref{thm:main}, we are reduced to show
that, for all $u,v\in \mathcal{S}(\H^n)$ we have 
  $$
  \|T_s(u,v)\|_{L^p}\lesssim \|\mathcal{L}_{s_1}u\|_{L^{p_1}}\|\mathcal{L}_{s_2}v\|_{L^{p_2}}.
  $$
Recall the definition of the square fractional integral
  $$
  \mathcal{D}_su(x):=\Big(\int_{\H^n}\frac{|u(xy^{-1})-u(x)|^2}{|y|^{Q+4s}}\,dy\Big)^{1/2}, \quad 0<s<1/2.
  $$
 Observe that, by Cauchy-Schwartz we have the poinwise estimate
 \begin{equation}
 \label{eq:CS}
 |T_s(u,v)(x)|\le  \mathcal{D}_{s_1}u(x) \mathcal{D}_{s_2}v(x), \qquad x\in \H^n, \qquad s=s_1+s_2, \qquad s_j\in (0,1/4).
\end{equation}
In view of \eqref{eq:CS}, H\"older's inequality and Theorem \ref{thm:pointwise} yield
$$
  \|T_s(u,v)\|_{L^p}\lesssim   \|\mathcal{D}_{s_1}u\|_{L^{p_1}}  \|\mathcal{D}_{s_2}u\|_{L^{p_2}}\le \Lambda(n,s_1)\Lambda(n,s_2) \|g_{\lambda_1}^*(\mathcal{L}^{s_1}u)\|_{L^{p_1}}\|g_{\lambda_2}^*(\mathcal{L}^{s_2}u)\|_{L^{p_2}},
  $$
  for any $p,p_1,p_2\in (0,\infty]$ with $\frac1p=\frac{1}{p_1}+\frac{1}{p_2}$, any $s_j\in (0,1/4)$ and any $\lambda_j<1+\frac{2s_j}{Q}$. 
By Theorem \ref{thm:gstar} and \eqref{eq:equiv} we conclude \eqref{eq:1}, namely
  $$
  \|T_s(u,v)\|_{L^p}\lesssim   \|\mathcal{D}_{s_1}u\|_{L^{p_1}}  \|\mathcal{D}_{s_2}u\|_{L^{p_2}}\lesssim\Lambda(n,s_1)\Lambda(n,s_2)\|\mathcal{L}_{s_1}u\|_{L^{p_1}}\|\mathcal{L}_{s_2}u\|_{L^{p_2}}
  $$
provided $\lambda_j$ is such that $\max\big\{1,\frac{2}{p_j}\big\}<\lambda_j<1+\frac{2s_j}{Q}$. The estimate \eqref{eq:2} follows from \eqref{eq:1} and \eqref{eq:equiv}.

The weighted estimate \eqref{eq:3} is proved similarly, using Theorem \ref{thm:gstarw}, which imposes the conditions $w_j\in A_{q_j}$, for $1<q_j<\min\big\{p_j,p_j\big(\frac12+\frac{s_j}{Q}\big)\big\}$. Nevertheless, the self-improving property of Muckenhoupt weights allows to relax the condition into $w_j\in A_{q_j}$, for $1<q_j=\min\big\{p_j,p_j\big(\frac12+\frac{s_j}{Q}\big)\big\}$. Since we always have $s_j/Q<1/2$, we obtain that $w_j\in A_{q_j}$, for $1<q_j=p_j\big(\frac12+\frac{s_j}{Q}\big)$, as desired. Finally, the estimate \eqref{eq:4} follows from  \eqref{eq:equiv}.
 

\end{document}